\documentclass[12pt,a4paper]{article}

\usepackage[T1]{fontenc}
\usepackage{lmodern}
\usepackage{amsmath,amssymb,amsthm}
\usepackage[margin=28mm]{geometry}
\usepackage{microtype}
\usepackage{cite}
\usepackage[hidelinks]{hyperref}

\hypersetup{
  pdftitle={Local and 2-local derivations of Block Lie algebras B(q)},
  pdfauthor={Jakhongir Ruzimboev and Bakhtiyor B. Yusupov},
  pdfkeywords={Block Lie algebra, local derivation, 2-local derivation, Laurent polynomial}
}

\newcommand{\C}{\mathbb C}
\newcommand{\Z}{\mathbb Z}
\newcommand{\B}{\mathcal B}

\DeclareMathOperator{\Der}{Der}
\DeclareMathOperator{\LDer}{LDer}
\DeclareMathOperator{\ad}{ad}

\newtheorem{theorem}{Theorem}
\newtheorem{proposition}[theorem]{Proposition}
\newtheorem{lemma}[theorem]{Lemma}

\theoremstyle{definition}
\newtheorem{definition}[theorem]{Definition}

\numberwithin{equation}{section}
\title{\bfseries Local and 2-local derivations of Block Lie algebras $\B(q)$}
\author{
Jakhongir Ruzimboev\thanks{New Uzbekistan University, Tashkent, Uzbekistan;
Tashkent International University of Financial Management
and Technologies, Tashkent, Uzbekistan.
Email: \href{mailto:jahongirrozimboyev5@gmail.com}{\nolinkurl{jahongirrozimboyev5@gmail.com}}.}
\and
Bakhtiyor Yusupov\thanks{V.I. Romanovskiy Institute of Mathematics,
Uzbekistan Academy of Sciences, Tashkent, Uzbekistan; Department of Applied Mathematics and Mathematical Physics, Urgench State University, Urgench, Uzbekistan.
Email: \href{mailto:baxtiyor_yusupov_93@mail.ru}{\nolinkurl{baxtiyor_yusupov_93@mail.ru}}.}
}

\date{}

\begin{document}

\maketitle

\begin{abstract}
We prove that every local or 2-local derivation of the Block
Lie algebra $\B(q)$, with both basis indices in $\Z$, is a
derivation for all $q\in\C$. The proofs use the known description
of ordinary derivations, Laurent polynomials, and coefficient
comparisons.
\end{abstract}



\section{Introduction}

Local derivations were introduced by Kadison~\cite{Kadison} and
independently studied by Larson and Sourour~\cite{LarsonSourour}.
A local derivation is a linear map that agrees at each element
with a derivation which may depend on that element.
In 1997, \v{S}emrl~\cite{Semrl} introduced 2-local derivations.
For such a map, each pair of elements admits one derivation
that agrees with the map at both elements; the map itself
is not assumed to be linear.
Both definitions lead to the question of when these local
conditions imply that the map is a derivation on the whole
algebra.

For finite-dimensional Lie algebras, Ayupov and
Kudaybergenov~\cite{AyupovKudaybergenov} proved that every
local derivation of a semisimple Lie algebra is a derivation.
They also gave examples of nilpotent Lie algebras admitting
local derivations which are not derivations.
The corresponding result for 2-local derivations of
semisimple Lie algebras was established by Ayupov,
Kudaybergenov and Rakhimov
\cite{AyupovKudaybergenovRakhimov}.
Alauadinov and Yusupov studied local derivations of
Schr\"odinger algebras
\cite{AlauadinovYusupovSchrodinger,
AlauadinovYusupovHigherSchrodinger}
and conformal Galilei algebras
\cite{AlauadinovYusupovGalilei}.
Local and 2-local derivations of solvable Leibniz algebras
were studied in~\cite{AyupovKhudoyberdiyevYusupov}.

For infinite-dimensional Lie algebras, Chen, Zhao and
Zhao~\cite{ChenZhaoZhao} proved that every local derivation
of a Witt algebra is a derivation.
Results on 2-local derivations of Witt algebras were obtained
by Yusupov~\cite{YusupovWitt}, Ayupov and
Yusupov~\cite{AyupovYusupov}, and Zhao, Chen and
Zhao~\cite{ZhaoChenZhaoTwoLocal}.
Related results for generalized Witt algebras and locally
simple Lie algebras were proved in
\cite{AyupovKudaybergenovYusupovWitt,
AyupovKudaybergenovYusupovSimple}.
Ruzimboev and Boranbaev~\cite{RuzimboevBoranbaev}
proved that every local superderivation of the modified
Witt-type Lie superalgebra is a superderivation.

Block introduced a class of infinite-dimensional Lie algebras
in~\cite{Block}.
Biderivations and transposed Poisson structures on Block Lie
algebras have been studied in
\cite{LiuGuoZhao,KaygorodovKhrypchenko}.
Some local and 2-local results are already known for
algebras related to the family $\mathcal B(q)$.
For $q=1$, Wu and Yang~\cite{WuYang} proved that every
local derivation of the subalgebra spanned by
$\{L_{m,i}:m\in\mathbb Z,\ i\geq0\}$ is a derivation.
Chen and Guo~\cite{ChenGuoBlockTwoLocal} proved the
corresponding result for 2-local derivations in an
isomorphic presentation of this subalgebra.
Both works restrict the second basis index to
nonnegative integers.

The zero parameter is related to the Virasoro-like algebra.
Indeed, $[\mathcal B(0),\mathcal B(0)]$ is the
Virasoro-like algebra, and $\mathcal B(0)$ is its direct
sum with the one-dimensional centre $\mathbb C L_{0,0}$.
Tang, Xiao and Wang~\cite{TangXiaoWangVirasoroLike}
proved that every 2-local derivation of the
Virasoro-like algebra is a derivation.

In this paper, we consider $\mathcal B(q)$ with both
basis indices in $\mathbb Z$ and with an arbitrary
complex parameter $q$.
We prove that every local derivation and every
2-local derivation of $\mathcal B(q)$ is a derivation.
The proofs use Laurent polynomials, coefficient comparisons,
and the description of ordinary derivations in
\cite{KaygorodovKhudoyberdiyevShermatova}.
We treat the cases $q\notin\mathbb Z$,
$q\in\mathbb Z\setminus\{0\}$ and $q=0$ separately.
Section~\ref{sec:local} deals with local derivations, and
Section~\ref{sec:bq-two-local} proves the result for
2-local derivations.

\section{Preliminaries}\label{sec:preliminaries}

All vector spaces and algebras are over $\C$.
All sums indexed by basis elements are finite.
We write $\delta_{a,b}$ for the Kronecker symbol and
$\C^*=\C\setminus\{0\}$.

A derivation of a Lie algebra $\mathfrak g$ is a linear map
$D:\mathfrak g\to\mathfrak g$ satisfying
\begin{equation*}
D([x,y])=[D(x),y]+[x,D(y)],
\qquad x,y\in\mathfrak g.
\end{equation*}
The space of derivations is denoted by $\Der(\mathfrak g)$.
For $u\in\mathfrak g$, the inner derivation $\ad u$ is
defined by $\ad u(x)=[u,x]$.

\begin{definition}\label{def:local-derivation}
A linear map $\Delta:\mathfrak g\to\mathfrak g$ is called a
\emph{local derivation} if, for every $x\in\mathfrak g$,
there exists $D_x\in\Der(\mathfrak g)$ such that
$\Delta(x)=D_x(x)$.
The space of local derivations is denoted by
$\LDer(\mathfrak g)$.
\end{definition}

\begin{definition}\label{def:two-local-derivation}
A map $\Delta:\mathfrak g\to\mathfrak g$, not necessarily
linear, is called a \emph{2-local derivation} if, for every
$x,y\in\mathfrak g$, there exists
$D_{x,y}\in\Der(\mathfrak g)$ such that
\begin{equation*}
D_{x,y}(x)=\Delta(x),\qquad
D_{x,y}(y)=\Delta(y).
\end{equation*}
\end{definition}

Every 2-local derivation is homogeneous.
Indeed, applying the definition to $(x,\lambda x)$ gives
\begin{equation*}
\Delta(\lambda x)
=D_{x,\lambda x}(\lambda x)
=\lambda D_{x,\lambda x}(x)
=\lambda\Delta(x),
\qquad \lambda\in\C.
\end{equation*}
In particular, $\Delta(0)=0$.

If $\Delta$ is a local or 2-local derivation and
$D\in\Der(\mathfrak g)$, then $\Delta-D$ is a map of the
same kind. This follows by using $D_x-D$ or $D_{x,y}-D$,
respectively.
Moreover, an isomorphism $\Phi:\mathfrak g\to\mathfrak h$
preserves derivations, local derivations and 2-local
derivations under conjugation $T\mapsto\Phi T\Phi^{-1}$.

For $q\in\C$, the Block Lie algebra $\B(q)$ has basis
$\{L_{m,i}:m,i\in\Z\}$ and bracket
\begin{equation}\label{eq:block-bracket}
[L_{m,i},L_{n,j}]
=\bigl(n(i+q)-m(j+q)\bigr)L_{m+n,i+j},
\qquad m,n,i,j\in\Z.
\end{equation}
Define the linear maps $D_1,D_2$ by
\begin{equation*}
D_1(L_{m,i})=iL_{m,i},
\qquad
D_2(L_{m,i})=mL_{m,i}.
\end{equation*}
For $q\in\Z$, also define
\begin{equation*}
\begin{aligned}
D_3(L_{m,i})&=mL_{m,i-q},\\
D_4(L_{m,i})&=(i+q)L_{m,i-q},\\
D_5(L_{m,i})&=\delta_{m,0}\delta_{i,-2q}L_{0,-q}.
\end{aligned}
\end{equation*}
When $q=0$, we have $D_3=D_2$ and $D_4=D_1$.
We use the following description of the derivation algebra;
see \cite[Proposition~5]{KaygorodovKhudoyberdiyevShermatova}.

\begin{proposition}\label{prop:derivations}
For every $q\in\C$, we have the following direct sum decomposition
as vector spaces:
\begin{equation*}
\Der(\B(q))=
\begin{cases}
\ad(\B(q))\oplus\C D_1,
    &q\notin\Z,\\
\ad(\B(q))\oplus\C D_1\oplus\C D_3
    \oplus\C D_4\oplus\C D_5,
    &q\in\Z\setminus\{0\},\\
\ad(\B(0))\oplus\C D_1\oplus\C D_2\oplus\C D_5,
    &q=0.
\end{cases}
\end{equation*}
\end{proposition}

Identify the vector space $\B(q)$ with the Laurent polynomial
algebra $R=\C[z^{\pm1},t^{\pm1}]$ by $L_{m,i}=z^mt^i$.
Put $E_z=z\partial_z$ and $E_t=t\partial_t$.
Thus
\begin{equation*}
1=L_{0,0},\quad z=L_{1,0},\quad t=L_{0,1},
\qquad
E_zL_{m,i}=mL_{m,i},\quad E_tL_{m,i}=iL_{m,i}.
\end{equation*}
In this notation, the Lie bracket is
\begin{equation}\label{eq:laurent-bracket}
[f,g]_q=(E_tf+qf)E_zg-(E_zf)(E_tg+qg),
\qquad f,g\in R.
\end{equation}
Applying the right-hand side to two monomials gives
\eqref{eq:block-bracket}.
In particular, $D_1=E_t$ and $D_2=E_z$.

For $q\in\Z$, let $\ell_q:R\to\C$ be the linear functional
which extracts the coefficient of $z^0t^{-2q}$.
Then
\begin{equation*}
D_3(f)=t^{-q}E_zf,\qquad
D_4(f)=t^{-q}(E_tf+qf),\qquad
D_5(f)=\ell_q(f)t^{-q}.
\end{equation*}
The element $t^{-q}=L_{0,-q}$ is central by
\eqref{eq:block-bracket}.

For $p=(a,b)\in(\C^*)^2$, Laurent polynomials and their
formal partial derivatives can be evaluated at $p$.
We write $f_z=\partial_zf$ and $f_t=\partial_tf$; hence
$(E_zf)(p)=af_z(p)$ and $(E_tf)(p)=bf_t(p)$.
The notation $T(f)(p)$ means that we first apply $T$ to $f$
and then evaluate the resulting Laurent polynomial at $p$.

We will use two elementary facts: a Laurent polynomial in $t$
which vanishes at every $b\in\C^*$ is zero, and a Laurent
polynomial in $z,t$ which vanishes on $(\C^*)^2$ is zero.
Both statements follow by multiplying by suitable powers
of the variables and applying the corresponding facts for
ordinary polynomials.

\section{Local derivations of
\texorpdfstring{$\B(q)$}{B(q)}}\label{sec:local}

In this section, we study local derivations of $\B(q)$
in the three cases $q\notin\Z$, $q\in\Z\setminus\{0\}$,
and $q=0$.

\subsection{The case
\texorpdfstring{$q\notin\Z$}{q not in Z}}
\label{subsec:nonintegral}

Throughout this subsection, $q\in\C\setminus\Z$.
In particular, $q\ne0,-1$ and $k+q\ne0$ for all $k\in\Z$.
Proposition~\ref{prop:derivations} gives
\begin{equation}\label{eq:nonintegral-action}
\begin{aligned}
D(f)&=[u,f]_q+dE_tf\\
&=(E_tu+qu)E_zf+(d-E_zu)E_tf-q(E_zu)f
\end{aligned}
\end{equation}
for every derivation $D$, with $u\in R$ and $d\in\C$.

\begin{lemma}\label{lem:nonintegral-normalization}
For every $\Delta\in\LDer(\B(q))$, there exist $u\in R$
and $\kappa\in\C$ such that the local derivation
$T=\Delta-\ad u-\kappa D_1$ satisfies $T(1)=T(t)=0$.
\end{lemma}

\begin{proof}
Locality at $1$ and \eqref{eq:nonintegral-action} give
$\Delta(1)=[u,1]_q$ for some $u\in R$.
Hence $\Delta_0=\Delta-\ad u$ satisfies $\Delta_0(1)=0$.

For each $\lambda\in\C^*$, apply locality to $t+\lambda1$.
Since $\Delta_0(t+\lambda1)=\Delta_0(t)$,
formula~\eqref{eq:laurent-bracket} gives
\begin{equation*}
\Delta_0(t)
=-\bigl((q+1)t+q\lambda\bigr)E_zv_\lambda+d_\lambda t.
\end{equation*}
Here $v_\lambda\in R$ and $d_\lambda\in\C$ may depend on
$\lambda$. Group terms by powers of $z$:
\begin{equation*}
\Delta_0(t)=\sum_rz^rF_r(t),
\qquad
v_\lambda=\sum_rz^rU_{r,\lambda}(t).
\end{equation*}
The $F_r$ are fixed, independently of $\lambda$.
Since $E_z$ multiplies the $z^r$ term by $r$,
comparison gives
\begin{equation}\label{eq:nonintegral-parameter-factor}
\begin{aligned}
F_r(t)&=-r\bigl((q+1)t+q\lambda\bigr)U_{r,\lambda}(t),
&&r\ne0,\\
F_0(t)&=d_\lambda t.
\end{aligned}
\end{equation}
For any $b\ne0$, take $\lambda=-(q+1)b/q$.
This is nonzero, and evaluation at $t=b$ gives $F_r(b)=0$
for $r\ne0$. Therefore $F_r=0$ whenever $r\ne0$.
It follows that $\Delta_0(t)=\kappa t$ for some $\kappa\in\C$.
Subtracting $\kappa D_1$ proves the assertion.
\end{proof}

\begin{lemma}[Action on the basis]
\label{lem:nonintegral-action}
Let $T\in\LDer(\B(q))$ satisfy $T(1)=T(t)=0$.
Then there exist coefficients $c_{r,k}\in\C$,
only finitely many of which are nonzero, such that
\begin{equation}\label{eq:nonintegral-basis}
T(L_{m,i})=m\sum_{r,k}c_{r,k}L_{m+r,i+k},
\qquad m,i\in\Z.
\end{equation}
The coefficients $c_{r,k}$ are independent of $m$ and $i$.
\end{lemma}

\begin{proof}
Fix $f\in R$ and $p=(a,b)\in(\C^*)^2$, and write
$f_z=\partial_zf$ and $f_t=\partial_tf$. Define
\begin{equation}\label{eq:local-first-jet}
h=f-f(p)-f_z(p)(z-a)-f_t(p)(t-b).
\end{equation}
Then $h(p)=0$, $h_z=f_z-f_z(p)$ and $h_t=f_t-f_t(p)$,
so $h_z(p)=h_t(p)=0$. Consequently,
$(E_zh)(p)=a h_z(p)=0$ and $(E_th)(p)=b h_t(p)=0$.
By \eqref{eq:nonintegral-action}, these equalities imply
$D(h)(p)=0$ for every derivation $D$.
Locality gives $T(h)=D_h(h)$ for some derivation $D_h$,
and hence $T(h)(p)=0$.

Since $f(p)$, $f_z(p)$ and $f_t(p)$ are scalars,
linearity of $T$ and the equalities $T(1)=T(t)=0$ give
$T(h)=T(f)-f_z(p)T(z)$. Evaluating at $p$, we obtain
$T(f)(p)=f_z(p)T(z)(p)$. As $p$ is arbitrary,
the Laurent polynomial $T(f)-T(z)\partial_zf$ vanishes
on $(\C^*)^2$ and is therefore zero. Indeed, multiplying
by a suitable monomial reduces this assertion to
the corresponding fact for ordinary polynomials.
Thus $T(f)=T(z)\partial_zf$ for every $f\in R$.

Expand the fixed Laurent polynomial $z^{-1}T(z)$ as
\begin{equation*}
z^{-1}T(z)=\sum_{r,k}c_{r,k}z^rt^k,
\end{equation*}
where only finitely many coefficients are nonzero.
Applying the preceding identity to $L_{m,i}=z^mt^i$ yields
\begin{equation*}
\begin{aligned}
T(L_{m,i})
&=T(z)\partial_z(z^mt^i)
 =mz^mt^i\bigl(z^{-1}T(z)\bigr)\\
&=m\sum_{r,k}c_{r,k}z^{m+r}t^{i+k}
 =m\sum_{r,k}c_{r,k}L_{m+r,i+k}.
\end{aligned}
\end{equation*}
This proves \eqref{eq:nonintegral-basis}.
The coefficients are independent of $m$ and $i$ because
they come from the fixed polynomial $z^{-1}T(z)$.
\end{proof}

\begin{lemma}\label{lem:nonintegral-inner}
Let $T\in\LDer(\B(q))$ satisfy
\eqref{eq:nonintegral-basis}. Then $c_{r,k}=0$ for $r\ne0$
and $T$ is inner.
\end{lemma}

\begin{proof}
Fix $r\ne0$ and $k\in\Z$.
Locality at $L_{r,k}$ gives
$w=\sum_{a,b}w_{a,b}L_{a,b}$ and $d\in\C$ with
\begin{equation*}
T(L_{r,k})=[w,L_{r,k}]_q+dkL_{r,k}.
\end{equation*}
Compare the coefficient of $L_{2r,2k}$.
On the left, \eqref{eq:nonintegral-basis} gives $rc_{r,k}$.
On the right, the outer term contributes nothing.
In the bracket, only the $L_{r,k}$ term of $w$ can produce
$L_{2r,2k}$. Thus \eqref{eq:block-bracket} gives
\begin{equation*}
rc_{r,k}
=w_{r,k}\bigl(r(k+q)-r(k+q)\bigr)=0.
\end{equation*}
Hence $c_{r,k}=0$ for $r\ne0$.

Write $c_k=c_{0,k}$ and set
$v=\sum_k c_k(k+q)^{-1}L_{0,k}$.
This is a finite sum with nonzero denominators.
For every $m,i\in\Z$,
\begin{equation*}
[v,L_{m,i}]_q
=\sum_k\frac{c_k}{k+q}m(k+q)L_{m,i+k}
=m\sum_kc_kL_{m,i+k}=T(L_{m,i}).
\end{equation*}
Therefore $T=\ad v$.
\end{proof}

\subsection{The case
\texorpdfstring{$q\in\Z\setminus\{0\}$}{q in Z, q nonzero}}
\label{subsec:integral}

We first reduce to positive $q$.
The linear bijection
\begin{equation*}
\Phi_q:\B(q)\longrightarrow\B(-q),
\qquad
\Phi_q(L_{m,i})=-L_{m,-i},
\end{equation*}
is an isomorphism. Indeed, using the bracket with parameter
$-q$ in the target algebra, we obtain
\begin{equation*}
\begin{aligned}
\bigl[\Phi_q(L_{m,i}),\Phi_q(L_{n,j})\bigr]_{-q}
&=\bigl(n(-i-q)-m(-j-q)\bigr)L_{m+n,-i-j}\\
&=\Phi_q([L_{m,i},L_{n,j}]_q).
\end{aligned}
\end{equation*}
Conjugation by $\Phi_q$ preserves derivations and local
derivations. It therefore suffices to consider $q>0$.

For the rest of this subsection, fix a positive integer $q$.
Let $\ell_q(f)$ denote the coefficient of
$s=t^{-2q}=L_{0,-2q}$ in $f$, and put
\begin{equation*}
c=t^{-q}=L_{0,-q},
\qquad V_q=\ker\ell_q.
\end{equation*}
Thus $c$ is central, $R=V_q\oplus\C s$, and $V_q$ is the
span of all basis elements except $s$.
Proposition~\ref{prop:derivations} gives
\begin{equation}\label{eq:integral-action}
\begin{aligned}
D(f)={}&[u,f]_q+\alpha E_tf+\beta t^{-q}E_zf\\
&+\gamma t^{-q}(E_tf+qf)+\delta\ell_q(f)c.
\end{aligned}
\end{equation}

\begin{lemma}\label{lem:integral-normalization}
For every $\Delta\in\LDer(\B(q))$, there exists
$D\in\Der(\B(q))$ such that $T=\Delta-D$ satisfies
$T(1)=T(t)=T(c)=0$.
\end{lemma}

\begin{proof}
Formula~\eqref{eq:integral-action} gives $D(c)=-q\alpha c$
for every derivation. Thus $\Delta(c)=\rho c$ for some
$\rho\in\C$, and
$\Delta_0=\Delta+(\rho/q)D_1$ satisfies $\Delta_0(c)=0$.
Locality at $1$ gives
$\Delta_0(1)=[u,1]_q+\gamma qc$.
Since $c$ is central, $D_4(1)=qc$ and $D_4(c)=0$,
the local derivation $T=\Delta_0-\ad u-\gamma D_4$
satisfies $T(1)=T(c)=0$.

It remains to show $T(t)=0$.
For each $\lambda\ne0$, apply locality to $t+\lambda1$.
Formula~\eqref{eq:integral-action} yields
\begin{equation}\label{eq:integral-normalization-t}
\begin{aligned}
T(t)={}&-\bigl((q+1)t+q\lambda\bigr)E_zu_\lambda
+\alpha_\lambda t\\
&+\gamma_\lambda\bigl((q+1)t^{1-q}+q\lambda t^{-q}\bigr).
\end{aligned}
\end{equation}
Write $T(t)=\sum_rz^rF_r(t)$.
The terms involving $\alpha_\lambda$ and $\gamma_\lambda$
have $z$-degree zero. For $r\ne0$, coefficient comparison
therefore gives the first identity in
\eqref{eq:nonintegral-parameter-factor}.
The evaluation argument following that equation shows
that $F_r=0$ for $r\ne0$, since $q\ne0,-1$.
Thus $T(t)=F_0(t)$.

The $z^0$ coefficient of
\eqref{eq:integral-normalization-t} is
\begin{equation*}
F_0(t)=\alpha_\lambda t
+\gamma_\lambda(q+1)t^{1-q}
+\gamma_\lambda q\lambda t^{-q}.
\end{equation*}
The powers $1$, $1-q$ and $-q$ are distinct.
Comparing the last two coefficients for $\lambda=1,2$
gives $\gamma_1=\gamma_2$ and $\gamma_1=2\gamma_2$.
Therefore $\gamma_1=\gamma_2=0$, and $T(t)=\kappa t$
for some $\kappa\in\C$.

Finally, $T(t+c)=T(t)$.
The $z^0$ coefficient of its local derivation gives
\begin{equation*}
\kappa t=\alpha(t-qt^{-q})+\gamma(q+1)t^{1-q}.
\end{equation*}
Comparison at $t^{-q}$ gives $\alpha=0$, and comparison
at $t$ gives $\kappa=0$. Hence $T(t)=0$.
\end{proof}

\begin{lemma}\label{lem:integral-basis}
Let $T\in\LDer(\B(q))$ satisfy $T(1)=T(t)=0$.
There is a finite family of coefficients $c_{r,k}\in\C$,
independent of $m,i$, such that
\begin{equation}\label{eq:integral-basis}
T(L_{m,i})=m\sum_{r,k}c_{r,k}L_{m+r,i+k},
\qquad (m,i)\ne(0,-2q).
\end{equation}
\end{lemma}

\begin{proof}
We adapt the proof of Lemma~\ref{lem:nonintegral-action}
to $V_q$. Fix $f\in V_q$ and $p=(a,b)\in(\C^*)^2$,
and define $h$ by \eqref{eq:local-first-jet}.
Since $q>0$, the elements $1,z,t$ belong to $V_q$.
Hence $h\in V_q$, so $\ell_q(h)=0$.
As before, $h(p)=h_z(p)=h_t(p)=0$.

In \eqref{eq:integral-action}, the inner term and the terms
involving $\alpha,\beta,\gamma$ therefore vanish at $p$.
The remaining term is zero because $\ell_q(h)=0$.
Thus locality gives $T(h)(p)=0$.
The rest of the pointwise argument in
Lemma~\ref{lem:nonintegral-action} now gives
$T(f)=T(z)\partial_zf$ for every $f\in V_q$.
Expanding $z^{-1}T(z)=\sum_{r,k}c_{r,k}z^rt^k$
and applying this identity to $f=L_{m,i}\in V_q$
proves \eqref{eq:integral-basis}.
\end{proof}

\begin{lemma}\label{lem:integral-inner}
Let $T\in\LDer(\B(q))$ satisfy \eqref{eq:integral-basis}.
Then $c_{r,k}=0$ for $r\ne0$.
Moreover, there exist $v\in R$ and $\beta\in\C$ such that
$T(f)=[v,f]_q+\beta D_3(f)$ for all $f\in V_q$.
\end{lemma}

\begin{proof}
Use the coefficient comparison from the proof of
Lemma~\ref{lem:nonintegral-inner}.
For $r\ne0$, apply locality to $L_{r,k}$ and compare
the coefficient of $L_{2r,2k}$.
The additional maps $D_3,D_4$ preserve the first index $r$,
while $D_5(L_{r,k})=0$.
Thus they contribute nothing to this coefficient,
and the same comparison gives $rc_{r,k}=0$.
Hence $c_{r,k}=0$ for $r\ne0$.

Write $c_k=c_{0,k}$.
The construction of the inner element in
Lemma~\ref{lem:nonintegral-inner} applies to all $k\ne-q$;
the remaining coefficient is represented by $D_3$.
Explicitly, set
\begin{equation*}
v=\sum_{k\ne-q}\frac{c_k}{k+q}L_{0,k},
\qquad \beta=c_{-q}.
\end{equation*}
All denominators in this finite sum are nonzero.
Using \eqref{eq:block-bracket}, for $(m,i)\ne(0,-2q)$ we get
\begin{equation*}
\begin{aligned}
\bigl[v,L_{m,i}\bigr]_q+\beta D_3(L_{m,i})
&=m\sum_{k\ne-q}c_kL_{m,i+k}
  +mc_{-q}L_{m,i-q}\\
&=m\sum_kc_kL_{m,i+k}=T(L_{m,i}).
\end{aligned}
\end{equation*}
This proves the assertion on $V_q$.
\end{proof}

\begin{lemma}\label{lem:integral-central}
If $P\in\LDer(\B(q))$ and $P|_{V_q}=0$, then $P=\delta D_5$
for some $\delta\in\C$.
\end{lemma}

\begin{proof}
Put $H=P(s)$. Since $R=V_q\oplus\C s$,
$P(f)=\ell_q(f)H$ for all $f\in R$.
We show that $H$ is a scalar multiple of $c$.

Write $H=\sum_rz^rH_r(t)$.
Locality at $s$ and \eqref{eq:integral-action} give
\begin{equation*}
H=qt^{-2q}E_zu-2q\alpha t^{-2q}
-q\gamma t^{-3q}+\delta t^{-q}.
\end{equation*}
The $z^0$ coefficient of $E_zu$ is zero, so
\begin{equation*}
H_0(t)\in\operatorname{span}\{t^{-2q},t^{-3q},t^{-q}\}.
\end{equation*}

For $\lambda\ne0$, we have $P(s+\lambda t)=H$.
Applying locality once more gives
\begin{equation}\label{eq:integral-last-value}
\begin{aligned}
H={}&\bigl(qt^{-2q}-\lambda(q+1)t\bigr)E_zu_\lambda\\
&+\alpha_\lambda(-2qt^{-2q}+\lambda t)\\
&+\gamma_\lambda\bigl(-qt^{-3q}
  +\lambda(q+1)t^{1-q}\bigr)+\delta_\lambda t^{-q}.
\end{aligned}
\end{equation}
The powers $1$ and $1-q$ do not belong to
$\{-q,-2q,-3q\}$, because $q>0$.
Comparison of the $t$ and $t^{1-q}$ coefficients in the
$z^0$ part of \eqref{eq:integral-last-value} therefore gives
$\alpha_\lambda=\gamma_\lambda=0$.
Hence $H_0=\delta c$ for some fixed $\delta\in\C$.

For $r\ne0$, write
$u_\lambda=\sum_rz^rU_{r,\lambda}(t)$.
The $z^r$ coefficient of \eqref{eq:integral-last-value} is
\begin{equation*}
H_r(t)=r\bigl(qt^{-2q}-\lambda(q+1)t\bigr)U_{r,\lambda}(t).
\end{equation*}
We use the same evaluation argument as in the proof of
Lemma~\ref{lem:nonintegral-normalization}, now choosing
$\lambda=\frac{q}{q+1}b^{-(2q+1)}$ for each $b\ne0$.
The factor in parentheses vanishes at $t=b$,
so $H_r=0$ whenever $r\ne0$.
Thus $H=\delta c$, and therefore $P=\delta D_5$.
\end{proof}

\subsection{The case \texorpdfstring{$q=0$}{q=0}}
\label{subsec:zero}

Here the central element is $1=L_{0,0}$,
and $\ell_0(f)$ is the constant coefficient of $f$.
Proposition~\ref{prop:derivations} gives
\begin{equation}\label{eq:zero-action}
D(f)=[u,f]_0+\alpha E_tf+\beta E_zf+\delta\ell_0(f)1.
\end{equation}
In this case, $[u,f]_0=E_tu\,E_zf-E_zu\,E_tf$.
We first describe a local derivation whose value at $1$
has been removed by subtracting a multiple of $D_5$.

\begin{lemma}\label{lem:zero-basis}
Let $T\in\LDer(\B(0))$ satisfy $T(1)=0$.
There exist $A,B\in R$ such that
\begin{equation}\label{eq:zero-vector-form}
T(f)=A E_zf+B E_tf,
\qquad f\in R.
\end{equation}
In particular, if
$A=\sum_{r,k}a_{r,k}z^rt^k$ and
$B=\sum_{r,k}b_{r,k}z^rt^k$, then
\begin{equation}\label{eq:zero-basis}
T(L_{m,i})=\sum_{r,k}(ma_{r,k}+ib_{r,k})L_{m+r,i+k}.
\end{equation}
Both coefficient families are finite and independent of $m,i$.
\end{lemma}

\begin{proof}
We adapt the pointwise argument in the proof of
Lemma~\ref{lem:nonintegral-action}.
Here the inner and Euler terms in
\eqref{eq:zero-action} involve only first derivatives,
so the condition $h(p)=0$ is unnecessary.

Take $f\in\ker\ell_0$ and $p=(a,b)\in(\C^*)^2$, and set
\begin{equation*}
h=f-f_z(p)z-f_t(p)t.
\end{equation*}
Then $\ell_0(h)=0$ and $h_z(p)=h_t(p)=0$.
Thus every term in \eqref{eq:zero-action} vanishes at $p$
when applied to $h$, and locality gives $T(h)(p)=0$.
Linearity therefore yields
\begin{equation*}
T(f)(p)=f_z(p)T(z)(p)+f_t(p)T(t)(p).
\end{equation*}
Since this holds for every $p$, we obtain
$T(f)=T(z)\partial_zf+T(t)\partial_tf$.

Put $A=z^{-1}T(z)$ and $B=t^{-1}T(t)$.
This proves \eqref{eq:zero-vector-form} on $\ker\ell_0$.
For arbitrary $f$, write $f=f_0+\ell_0(f)1$ with
$f_0\in\ker\ell_0$.
Since $T(1)=E_z1=E_t1=0$, the same identity holds on all of $R$.
Substitution of $f=z^mt^i$ proves \eqref{eq:zero-basis}.
\end{proof}

\begin{lemma}\label{lem:zero-inner}
Let $T\in\LDer(\B(0))$ satisfy
\eqref{eq:zero-vector-form}.
Then there exist $u\in R$ and $\alpha,\beta\in\C$ such that
$T=\ad u+\alpha D_1+\beta D_2$.
\end{lemma}

\begin{proof}
Use the coefficient families in \eqref{eq:zero-basis}.
We apply the comparison from the proof of
Lemma~\ref{lem:nonintegral-inner} to each
$(r,k)\ne(0,0)$.
The coefficient of $L_{2r,2k}$ in $T(L_{r,k})$ is now
$ra_{r,k}+kb_{r,k}$.
The maps $D_1,D_2$ give only multiples of $L_{r,k}$,
and $D_5(L_{r,k})=0$.
Only the $L_{r,k}$ term of the inner element can contribute,
and its contribution is zero because
$[L_{r,k},L_{r,k}]_0=0$.
Therefore
\begin{equation}\label{eq:zero-coefficients}
ra_{r,k}+kb_{r,k}=0,
\qquad (r,k)\ne(0,0).
\end{equation}

Define
\begin{equation*}
u=\sum_{\substack{r,k\in\Z\\k\ne0}}
       \frac{a_{r,k}}{k}L_{r,k}
  -\sum_{r\ne0}\frac{b_{r,0}}{r}L_{r,0}.
\end{equation*}
The sum is finite. For $k\ne0$,
\eqref{eq:zero-coefficients} gives
$b_{r,k}=-ra_{r,k}/k$; for $k=0$ and $r\ne0$, it gives
$a_{r,0}=0$. Consequently,
\begin{equation*}
E_tu=A-a_{0,0},
\qquad -E_zu=B-b_{0,0}.
\end{equation*}
Using \eqref{eq:zero-vector-form}, we conclude that
\begin{equation*}
\begin{aligned}
T(f)
&=(E_tu+a_{0,0})E_zf+(-E_zu+b_{0,0})E_tf\\
&=[u,f]_0+b_{0,0}D_1(f)+a_{0,0}D_2(f).
\end{aligned}
\end{equation*}
This proves the lemma.
\end{proof}

The preceding lemmas give the following theorem.

\begin{theorem}\label{thm:main}
For every $q\in\C$, every local derivation on $\B(q)$
is a derivation. Equivalently,
\begin{equation*}
\LDer(\B(q))=\Der(\B(q)).
\end{equation*}
\end{theorem}

\begin{proof}
Every derivation is a local derivation.
For the converse, let $\Delta\in\LDer(\B(q))$.

If $q\notin\Z$, Lemma~\ref{lem:nonintegral-normalization}
gives $u\in R$ and $\kappa\in\C$ such that
$T=\Delta-\ad u-\kappa D_1$ satisfies $T(1)=T(t)=0$.
Lemmas~\ref{lem:nonintegral-action} and
\ref{lem:nonintegral-inner} give $T=\ad v$ for some $v\in R$.
Hence $\Delta=\ad(u+v)+\kappa D_1$ is a derivation.

If $q\in\Z\setminus\{0\}$, the isomorphism in
Subsection~\ref{subsec:integral} reduces the proof to $q>0$.
By Lemma~\ref{lem:integral-normalization},
$T=\Delta-D$ satisfies $T(1)=T(t)=T(c)=0$
for some $D\in\Der(\B(q))$.
Lemmas~\ref{lem:integral-basis} and \ref{lem:integral-inner}
give $v\in R$ and $\beta\in\C$ such that
$P=T-\ad v-\beta D_3$ vanishes on $V_q$.
Lemma~\ref{lem:integral-central} gives $P=\delta D_5$.
Therefore $\Delta=D+\ad v+\beta D_3+\delta D_5$
is a derivation.

Finally, let $q=0$.
By \eqref{eq:zero-action}, locality at $1$ gives
$\Delta(1)=\rho1$ for some $\rho\in\C$.
The local derivation $T=\Delta-\rho D_5$ satisfies $T(1)=0$.
Lemmas~\ref{lem:zero-basis} and \ref{lem:zero-inner} yield
$T=\ad u+\alpha D_1+\beta D_2$.
Thus $\Delta=\ad u+\alpha D_1+\beta D_2+\rho D_5$
is a derivation.
\end{proof}

\section{2-Local derivations of
\texorpdfstring{$\B(q)$}{B(q)}}\label{sec:bq-two-local}

In this section, we study 2-local derivations of $\B(q)$
in the three parameter cases.
For a 2-local derivation $T$, let $D_{x,y}$ denote a derivation
satisfying $D_{x,y}(x)=T(x)$ and $D_{x,y}(y)=T(y)$.
Its form is given by Proposition~\ref{prop:derivations},
with coefficients depending on the pair $(x,y)$.

Suppose that $f,g\in R$ have the same value and first
derivatives at $p\in(\C^*)^2$.
When $q\in\Z$, assume also that $\ell_q(f)=\ell_q(g)$.
As in Section~\ref{sec:local}, the derivation formulas give
$D(f-g)(p)=0$ for every $D\in\Der(\B(q))$.
By the linearity of $D_{f,g}$,
\begin{equation}\label{eq:bq-two-local-pointwise}
T(f)(p)-T(g)(p)=D_{f,g}(f-g)(p)=0.
\end{equation}

\subsection{The case
\texorpdfstring{$q\notin\Z$}{q not in Z}}
\label{subsec:bq-two-local-nonintegral}

\begin{lemma}\label{lem:bq-two-local-zero}
Let $q\notin\Z$. If a 2-local derivation $T$ satisfies
$T(t)=T(z)=0$, then $T=0$.
\end{lemma}

\begin{proof}
For $f\in R$, write $D_{t,f}=\ad u+\alpha D_1$.
Comparing $z$-coefficients in
\begin{equation*}
0=D_{t,f}(t)=-(q+1)tE_zu+\alpha t
\end{equation*}
gives $\alpha=0$ and $E_zu=0$. Hence
\begin{equation}\label{eq:bq-two-local-profile}
T(f)=F_f(t)E_zf,
\qquad F_f=(E_t+q)u\in\C[t^{\pm1}].
\end{equation}

Choose an integer $N\ge2$ such that
$(Nq+1)((N+1)q+1)\ne0$, and put $g=Az^N+Bt+C$,
where $A,B,C\in\C$.
If $A=0$, then \eqref{eq:bq-two-local-profile} gives $T(g)=0$.
For $A\ne0$, write $D_{z,g}=\ad u+\alpha D_1$,
where $u=\sum_{r,k}a_{r,k}L_{r,k}$.
Since
\begin{equation*}
0=D_{z,g}(z)
=\sum_{r,k}(k+q-rq)a_{r,k}L_{r+1,k},
\end{equation*}
we have
\begin{equation}\label{eq:bq-two-local-resonance}
(k+q-rq)a_{r,k}=0,
\qquad r,k\in\Z.
\end{equation}
In particular, $a_{0,k}=0$ for all $k$. Consequently,
\begin{equation*}
\begin{aligned}
T(g)
&=ANz^NF_g(t)\\
&=Aq(N-1)\sum_{r,k}ra_{r,k}L_{r+N,k}
  -B(q+1)\sum_{r,k}ra_{r,k}L_{r,k+1}\\
&\quad-Cq\sum_{r,k}ra_{r,k}L_{r,k}+\alpha Bt.
\end{aligned}
\end{equation*}
The four terms on the right have the following respective
values of $i-qm$:
\begin{equation*}
\begin{aligned}
k-q(r+N)&=-(N+1)q,&
(k+1)-qr&=1-q,\\
k-qr&=-q,&
1-q\cdot0&=1.
\end{aligned}
\end{equation*}
Since $-(N+1)q\notin\{1-q,-q,1\}$ and every nonzero term
of the first sum has $r+N\ne N$, comparison with
$ANz^NF_g(t)$ gives $Aq(N-1)ra_{r,k}=0$.
Thus $u=0$, and comparison of $z$-degrees gives
$T(g)=\alpha Bt=0$.

Finally, for $f\in R$ and $p=(a,b)\in(\C^*)^2$, set
\begin{equation*}
A=\frac{f_z(p)}{Na^{N-1}},\qquad
B=f_t(p),\qquad C=f(p)-Aa^N-Bb.
\end{equation*}
Then $g(p)=f(p)$, $g_z(p)=f_z(p)$ and $g_t(p)=f_t(p)$.
By \eqref{eq:bq-two-local-pointwise}, $T(f)(p)=T(g)(p)=0$.
As $p$ is arbitrary, $T(f)=0$, proving the lemma.
\end{proof}

\subsection{The case
\texorpdfstring{$q\in\Z\setminus\{0\}$}{q in Z, q nonzero}}
\label{subsec:bq-two-local-integral}

The isomorphism $\Phi_q:\B(q)\to\B(-q)$ from
Subsection~\ref{subsec:integral} preserves 2-local derivations
under conjugation. It therefore suffices to consider $q>0$.

\begin{lemma}\label{lem:bq-two-local-integral-zero}
Let $q$ be a positive integer and put $s=t^{-2q}$.
If a 2-local derivation $T$ satisfies $T(s)=T(z)=0$,
then $T=0$.
\end{lemma}

\begin{proof}
Put $c=t^{-q}$. For $f\in R$, write
$D_{s,f}=\ad u+\alpha D_1+\beta D_3+\gamma D_4+\delta D_5$.
Then
\begin{equation*}
0=D_{s,f}(s)=qsE_zu-2q\alpha s-q\gamma t^{-3q}+\delta c.
\end{equation*}
Separating the $z$-degrees and the three distinct powers of
$t$ gives $E_zu=0$ and $\alpha=\gamma=\delta=0$.
Thus \eqref{eq:bq-two-local-profile} holds, now with
$F_f=(E_t+q)u+\beta c$.

Consider
\begin{equation*}
g=Az^3+\lambda s+Bt+C+t^2,
\qquad A,\lambda,B,C\in\C.
\end{equation*}
If $A=0$, then \eqref{eq:bq-two-local-profile} gives $T(g)=0$.
For $A\ne0$, write
$D_{z,g}=\ad u+\alpha D_1+\beta D_3+\gamma D_4+\delta D_5$,
where $u=\sum_{r,k}a_{r,k}L_{r,k}$.
The equality
\begin{equation*}
0=D_{z,g}(z)
=\sum_{r,k}(k+q-rq)a_{r,k}L_{r+1,k}
  +(\beta+q\gamma)L_{1,-q}
\end{equation*}
gives \eqref{eq:bq-two-local-resonance} and
$\beta+q\gamma=0$.

As in the preceding proof, compare $D_{z,g}(g)$ with
$T(g)=3Az^3F_g(t)$.
The terms arising from $Az^3$ with first index different
from $3$ are
\begin{equation*}
2Aq\sum_{\substack{r,k\\r\ne0}}ra_{r,k}L_{r+3,k},
\qquad k-q(r+3)=-4q.
\end{equation*}
All remaining terms with first index different from $3$
have
\begin{equation*}
i-qm\in\{-3q,-2q,-q,1-q,2-q,1,2\}.
\end{equation*}
Since $q>0$, none of these values is $-4q$.
Coefficient comparison gives $a_{r,k}=0$ for $r\ne0$;
then \eqref{eq:bq-two-local-resonance} gives
$u=a_{0,-q}c$, which is central.

Using $\beta=-q\gamma$, we obtain
\begin{equation*}
\begin{aligned}
T(g)={}&-2Aq\gamma z^3c
 +\alpha(-2q\lambda s+Bt+2t^2)\\
&+\gamma\bigl(-q\lambda t^{-3q}+(q+1)Bt^{1-q}
              +qCc+(q+2)t^{2-q}\bigr)
 +\delta\lambda c.
\end{aligned}
\end{equation*}
Comparing the $z^0$ coefficients at $t^2$ and then
$t^{2-q}$ gives
\begin{equation*}
2\alpha=0,\qquad (q+2)\gamma=0,\qquad \beta=-q\gamma=0.
\end{equation*}
Hence $T(g)=\delta\lambda c=0$ by
\eqref{eq:bq-two-local-profile}.

For $f\in R$ and $p=(a,b)\in(\C^*)^2$, choose
\begin{equation*}
\begin{aligned}
\lambda&=\ell_q(f),\qquad A=\frac{f_z(p)}{3a^2},\\
B&=f_t(p)+2q\lambda b^{-2q-1}-2b,\\
C&=f(p)-Aa^3-\lambda b^{-2q}-Bb-b^2.
\end{aligned}
\end{equation*}
Then $g$ and $f$ have the same value and first derivatives
at $p$, and $\ell_q(g)=\ell_q(f)$.
Equation~\eqref{eq:bq-two-local-pointwise} gives
$T(f)(p)=T(g)(p)=0$, and therefore $T(f)=0$.
\end{proof}

\subsection{The case \texorpdfstring{$q=0$}{q = 0}}
\label{subsec:bq-two-local-zero}

\begin{lemma}\label{lem:bq-two-local-zero-linearity}
Every 2-local derivation on $\B(0)$ is linear.
\end{lemma}

\begin{proof}
Let $\Delta$ be a 2-local derivation, and put
$c=L_{0,0}$ and
$G=\operatorname{span}\{L_{m,i}:(m,i)\ne(0,0)\}$.
Then $\B(0)=G\oplus\C c$.
By Proposition~\ref{prop:derivations}, every derivation
preserves both summands; hence $\Delta(G)\subseteq G$
and $\Delta(c)=\rho c$ for some $\rho\in\C$.
The symmetric bilinear form $\langle x,y\rangle=\ell_0(xy)$
on $G$ is nondegenerate, since
\begin{equation*}
\langle L_{m,i},L_{n,j}\rangle
=\delta_{m+n,0}\delta_{i+j,0}.
\end{equation*}

For $x,y\in G$, direct calculation gives
\begin{equation*}
\begin{aligned}
\langle[u,x]_0,y\rangle+\langle x,[u,y]_0\rangle
&=\ell_0([u,xy]_0)=0,\\
\langle E_\nu x,y\rangle+\langle x,E_\nu y\rangle
&=\ell_0(E_\nu(xy))=0,\qquad \nu=z,t.
\end{aligned}
\end{equation*}
Since $D_5|_G=0$, every derivation is skew with respect
to this form. Applying this to $D_{x,y}$ yields
\begin{equation*}
\langle\Delta(x),y\rangle=-\langle x,\Delta(y)\rangle.
\end{equation*}
Thus, for $x,y,z\in G$,
\begin{equation*}
\begin{aligned}
\langle\Delta(x+y),z\rangle
&=-\langle x+y,\Delta(z)\rangle\\
&=\langle\Delta(x)+\Delta(y),z\rangle.
\end{aligned}
\end{equation*}
Nondegeneracy gives additivity, and homogeneity gives
linearity on $G$.

For $f=x+\lambda c$, with $x\in G$, the derivations
$D_{f,x}$ and $D_{f,c}$ determine the $G$ and central
components of $\Delta(f)$, respectively. Hence
\begin{equation*}
\Delta(x+\lambda c)=\Delta(x)+\lambda\rho c,
\qquad x\in G,\quad\lambda\in\C.
\end{equation*}
Therefore $\Delta$ is linear on $\B(0)$.
\end{proof}

The three cases give the following theorem.

\begin{theorem}\label{thm:bq-two-local-main}
For every $q\in\C$, every 2-local derivation on $\B(q)$
is a derivation.
\end{theorem}

\begin{proof}
Let $\Delta$ be a 2-local derivation on $\B(q)$.
If $q\notin\Z$, put $T=\Delta-D_{t,z}$.
Then $T(t)=T(z)=0$, so Lemma~\ref{lem:bq-two-local-zero}
gives $T=0$ and $\Delta=D_{t,z}$.

For $q\in\Z\setminus\{0\}$, use the reduction in
Subsection~\ref{subsec:bq-two-local-integral} to assume $q>0$.
Put $s=t^{-2q}$ and $T=\Delta-D_{s,z}$.
Since $T(s)=T(z)=0$, Lemma~\ref{lem:bq-two-local-integral-zero}
gives $T=0$ and $\Delta=D_{s,z}$.

If $q=0$, Lemma~\ref{lem:bq-two-local-zero-linearity}
shows that $\Delta$ is linear.
It is therefore a local derivation, and Theorem~\ref{thm:main}
completes the proof.
\end{proof}

\end{document}